\documentclass[12point]{article}
\usepackage[all, ps, dvips]{xy} 
\usepackage{color} 
\usepackage{amsmath}
\usepackage{amsfonts}
\usepackage{amssymb}
\usepackage{mathrsfs}
\usepackage{amsthm}

\newtheorem{theorem}{Theorem}
\newtheorem{lemma}[theorem]{Lemma}

\newtheorem{corollary}[theorem]{Corollary}

\newenvironment{remark}[1][Remark.]{\begin{trivlist}
\item[\hskip \labelsep {\bfseries #1}]}{\end{trivlist}}

\newcommand{\tens}[1]{
  \mathbin{\mathop{\otimes}\limits_{#1}}
}

\everymath{\displaystyle}

\begin{document}
\title{The Bernstein-Sato b-function for the complement of the open $SL_n$-orbit on a triple flag variety}
\author{Henry Scher}
\maketitle

\begin{abstract}
We calculate Bernstein-Sato b-functions for $f_{G^3}^\lambda$, a $SL_n$-invariant section of a line bundle on $SL_n/B \times SL_n/B \times \mathbb{P}^{n - 1}$ whose zero-set is the complement of the open $G$-diagonal orbit. The proof uses a similar calculation by Kashiwara of the b-function for $f^\lambda$, a $B^-$-semiinvariant section of a line bundle on $SL_n/B$ whose zero-set is the complement of the big Bruhat cell. 
\end{abstract}

\section{Introduction}
Let $X$ be an algebraic variety over $\mathbb{C}$, and let $f$ be an algebraic function on $X$. The Bernstein-Sato function of $f$ is defined by the monic generator of the ideal of $b(s)$ such that there is a differential operator $P(s)$ such that

\begin{equation}
P(s) f^{s + 1} = b(s) f^s
\end{equation}
where $P(s) \in \mathcal{D}(X)[s]$. There have been various generalizations to multiple functions. In this paper, the ideal will be defined by the following functional equation:

\begin{equation}
P_{m_1, m_2, \dots, m_n}(s_1, s_2, \dots, s_n) \prod_i f_i^{s_i + m_i} = b_{m_1, m_2, \dots, m_n}(s_1, s_2, \dots, s_n)
\end{equation}
where $P(s) \in \mathcal{D}(X)[s_1, s_2, \dots, s_n], m_i \in \mathbb{Z}$. Sabbah \cite{SABBAH} and Gyoja \cite{GYOJA} generalized Kashiwara's result about rationality of the roots of b-functions by showing that for each choice of $\{t_i\}$, there is an element in the b-function ideal that can be written as the product of hyperplanes (i.e. functions of the form $\sum_i a_i t_i + k_i$). 

Let $G = SL_n$, $B$ a Borel subgroup, and $\mathcal{F}$ be the corresponding flag variety. The action of $G$ on $\mathcal{F}$ has a unique open $B$-orbit; its complement is a hypersurface, so there are $B$-semiinvariant global sections $f^\lambda$ of $G$-equivariant line bundles (which, for each line bundle, are unique when they exist). The zero-sets of these sections are components of the hypersurface. In 1985, Kashiwara \cite{KASHIWARA} found the b-function for these sections using universal Verma modules.

Let $X = \mathcal{F} \times \mathcal{F} \times \mathbb{P}^{n - 1}$. Then the diagonal $G$-action on $X$ has an open orbit with a hypersurface complement; correspondingly, there are $G$-semiinvariant global sections (for a specific character) of $G$-equivariant line bundles $\mathcal{L}$ on $X$ whose zero-sets are components of the complement of this open orbit \cite{GINZBURGFINKELBERG}. It has been conjectured that the b-function of these sections could be found using techniques similar to those used by Kashiwara. This paper finds those b-functions.

We start in section 3 by defining the $G$-invariant global sections $f_{G}^\lambda$ of line bundles on $X$, and the relevant functional equation for the type of b-functions we wish to use. In section 4, we explain some of Kashiwara's techniques, including finding differential operators on $\mathcal{F}$, and use his results to find a relationship among his differential operators. In section 5, we use Kashiwara's differential operator to find the differential operator on $X$ that satisfies the functional equation, and a relationship between the b-function Kashiwara found for a different global section and the b-function for the global section examined in this paper, through a function $H(\lambda)$; we also show that there is a similar relationship between the differential operators for the $G$-invariant global sections. In section 6, we examine $H(\lambda)$ more closely and find what it is for some of the $G$-invariant global sections. In section 7, we use the relationships between the differential operators in order to prove one of the conjectures from Ginzburg and Finkelberg \cite{GINZBURGFINKELBERG}, that the b-functions can be factored into linear factors, and also prove some relationships between the b-functions. In section 8, we use the values given in section 6 with the relationships in section 5 to fully determine all of the b-functions.

\section{Notation}
\begin{enumerate}
\item $G = SL_n$

\item $B$ is a Borel subgroup of $G$.

\item $P$ is a mirabolic subgroup of $G$, that is, a subgroup of $G$ that fixes a line in $\mathbb{C}^n$. 

\item $\mathcal{F} = G/B$ is the flag variety of $G$, thought of as the moduli space of Borel subgroups.

\item $X = \mathcal{F} \times \mathcal{F} \times \mathbb{P}^{n - 1}$ is the moduli space of triples of two Borel subgroups and a mirabolic subgroup. 

\item $V_{\lambda_1}$ is the irreducible representation of $G$ with highest weight $\lambda_1$.

\item If $\lambda = (\lambda_1, \lambda_2, l)$, then $V_\lambda := V_{\lambda_1} \otimes V_{\lambda_2} \otimes Sym^l \mathbb{C}^n$ is an irreducible representation of $G^3 = G \times G \times G$.

\item $B_1 \times B_2 \times P$ is a triple of subgroups of $G$, with $B_1, B_2$ Borel and $P$ mirabolic.

\item $B_1' \times B_2' \times P'$ is a triple of subgroups of $G$, with $B_1', B_2'$ Borel and $P'$ dual-mirabolic (i.e. fixing a line in $\mathbb{C}^n$ with the dual action of $G$). 

\item $v_\lambda$ is a highest weight vector in $V_\lambda$ with respect to $B_1 \times B_2 \times P$. 

\item $v_{-\lambda}$ is a highest weight vector in $V_\lambda^*$ with respect to $B_1' \times B_2' \times P'$ such that $\langle v_{-\lambda}, v_\lambda \rangle = 1$.

\item $\Omega$ is the set of triples $\lambda$ such that there is a diagonally $G$-invariant vector $u_{-\lambda} \in V_\lambda^*$ with $\langle u_{-\lambda}, v_\lambda \rangle = 1$.

\item $u_\lambda \in V_\lambda$ is the diagonally $G$-invariant vector with $\langle v_{-\lambda}, u_\lambda \rangle = 1$.

\item $H(\lambda) = \frac{\langle v_{-\lambda}, v_\lambda\rangle \langle u_{-\lambda}, u_\lambda \rangle}{\langle v_{-\lambda}, u_\lambda \rangle \langle u_{-\lambda}, v_\lambda \rangle}$

\item $\omega$ is the standard representation $\mathbb{C}^n$ of $G$

\item $\wedge^i \omega_1$ is the $i$th fundamental representation of $SL_n$.

\item $\omega_j$ is the standard representation of $SL_j \subset G$. 
\end{enumerate}
\section{The relevant line bundle and section}
Choose two Borel subgroups $B_1, B_2$ and a mirabolic subgroup $P$ of $G = SL_n$, i.e. a subgroup $P$ is the stabilizer of a point in $\mathbb{P}^{n - 1}$; such a subgroup is called a mirabolic subgroup. This subgroup is the stabilizer of a unique point $x \in X$, so $X$ can be seen as the moduli space of such triples. We further require that they be in general position; that is, that $x$ lie in the open $G$-diagonal orbit. Let $\Gamma$ be the lattice of triples $(\lambda_1, \lambda_2, l)$ of pairs of integral weights and an integer; let $\Gamma_{\geq 0}$ be the subcone where the weights are dominant and the integer is nonnegative.

For $\lambda = (\lambda_1, \lambda_2, l) \in \Gamma_{\geq 0}$, define $V_\lambda = V_{\lambda_1} \otimes V_{\lambda_2} \otimes Sym^l \mathbb{C}^n$. Let $\Omega$ be the set of $\lambda \in \Gamma_{\geq 0}$ such that $V_\lambda^*$ has a nontrivial $G$-invariant element.

\begin{lemma}
$\Omega$ is a cone (i.e. it is closed under addition).
\end{lemma}
\begin{proof}
By the Borel-Weil theorem, the global sections of the line bundle $L_\lambda$ on $X = \mathcal{F} \times \mathcal{F} \times \mathbb{P}^{n - 1}$ form a representation of $G^3$ isomorphic to $V_\lambda^*$. As such, $V_\lambda^*$ contains a nontrivial $G$-invariant element under the diagonal action of $G$ if and only if there is a nontrivial global $G$-invariant section of $L_\lambda$. Then because $X$ is an irreducible variety (and as such has a homogeneous coordinate ring without zero divisors), if there is a nontrivial $G$-invariant element of $V_\lambda^*$ and another of $V_{\lambda'}^*$, then there necessarily is a nontrivial $G$-invariant element of $V_{\lambda + \lambda'}^*$ corresponding to the product of the two invariant sections in $L_{\lambda + \lambda'}$.
\end{proof}

For $\lambda \in \Gamma_{\geq 0}$, define $v_\lambda \in V_\lambda$ as a nonzero highest-weight element with respect to $B_1 \times B_2 \times P$. Because we have chosen our subgroups in general position, $v_\gamma^G \neq 0$ for all $\gamma \in \Gamma_{\geq 0}$.   

Define $X' = \mathcal{F} \times \mathcal{F} \times (\mathbb{P}^{n - 1})^\vee$ with a $G^3$-action, where the $G$-action on $(\mathbb{P}^{n - 1})^\vee$ comes from the action of $G$ on the dual of the usual representation. Then $X'$ corresponds to the moduli space of triples of two Borel subgroups $B_1', B_2'$ and a dual-mirabolic subgroup $P'$ (i.e. a mirabolic subgroup fixing an element of the dual of the standard representation). We then say that $x' \in X'$ is in general position if it is in the open $G$-orbit on $X'$. If we choose an invariant element $u_{-\lambda} \in V_\lambda^*$, then $\langle u_{-\lambda}, v_\lambda \rangle = 0$ is the identifying equation of a closed subvariety of the complement of the open orbit. Therefore $x'$ being in the open orbit is equivalent to the existence for any $\lambda \in \Omega$ of a $u_{-\lambda} \in V_\lambda^*$ that is $G$-invariant with $\langle u_{-\lambda}, v_\lambda\rangle = 1$ (by scaling). Finally, a point $(x, x') \in X \times X'$ is in general position if both $x$ and $x'$ are in general position and each subgroup and its primed counterpart are opposite. Equivalently (for similar reasoning to the above), for any $\lambda \in \Gamma_{\geq 0}$, we want there to be an element $v_{-\lambda} \in V_\lambda^*$ of highest weight with respect to $B_1' \times B_2' \times P'$ such that $\langle v_{-\lambda}, v_\lambda\rangle = 1$. Borel-Weil implies that the correspondence between $V_\lambda^*$ and $\Gamma(X, L_\lambda)$ can be given by $w \rightarrow [g \rightarrow \langle w, g v_\lambda\rangle]$. We therefore have two sections of $L_\lambda$, $f_{G^3}^\lambda$ and $f_{K}^\lambda$, such that $f_{G^3}^\lambda$ is diagonally $G$-invariant, while $f_{K}^\lambda$ is left $B_1' \times B_2' \times P'$ semiinvariant. The former subscript denotes that the section is $G$-invariant; the latter denotes that the section comes from the work of Kashiwara.

Because there is both a unique open $G$-orbit and a unique open $B_1' \times B_2' \times P'$-orbit on $X$, these sections are the unique ones with this property up to scaling. Then as products of semiinvariants are semiinvariant, and by evaluating at the identity on $G^3$, we can see that $f_{K}^\lambda f_{K}^\mu = f_{K}^{\lambda + \mu}$ and $f_{G^3}^\lambda f_{G^3}^\mu = f_{G^3}^{\lambda + \mu}$. 

The b-function we want to find is that of $f_{G^3}^\lambda$; to be more precise, we want to find solutions to the following functional equation:
\begin{equation}
P_{\mu} f_{G^3}^{\lambda + \mu} = b_\mu(\lambda) f_{G^3}^\lambda 
\end{equation}
where $P_\mu$ is a twisted differential operator on $\mathcal{F}$ (equivalently, a differential operator on $G$ such that for any right $B$-semiinvariant $f$, $P_\mu f$ is also right $B$-semiinvariant).

This corresponds to a b-function using the definition given by Gyoja \cite{GYOJA}, where the $f_i$ are $f_\lambda$ for $\lambda$ a generator of $\Omega$. We will describe those generators in section 5.

\section{Kashiwara's argument and the differential operator}\label{Kashiwarasection}
Let $G$ be any semisimple complex simply connected Lie group with Lie algebra $\mathfrak{g}$; let $B$ be a Borel subgroup and $B'$ be an opposite Borel subgroup with common torus $H = B \cap B'$. Then Kashiwara wanted to find the b-function of the $B'$-semiinvariant section $f_G^\lambda$ of the line bundle $L_\lambda$ over $\mathcal{F} = G/B$, using its cover (which we will also call $f_G^\lambda$ through abuse of notation) $f_G^{\lambda}: G \rightarrow \mathbb{C}, g \rightarrow \langle v_{-\lambda}, g v_{\lambda}\rangle$ where $\lambda$ is an integral dominant weight, $v_{\lambda} \in V_{\lambda}$ is a highest weight vector wrt $B$ (and therefore also with respect to $H$), and $v_{-\lambda} \in V_{\lambda}^*$ is a highest weight vector wrt $B'$ (and lowest weight wrt $H$). 

Let $R_+$ be the set of positive roots of $G = SL_n$ with respect to $B$, $\rho = \frac{1}{2} \sum_{\alpha \in R_+} \alpha$, and for any $\alpha \in R_+$, let $h_\alpha(\mu) = \langle \alpha, \mu\rangle$. 

Define $\mathcal{D}^\mu(\mathcal{F})$ as the set of differential operators that twist by $-\mu$. In other words, if $P \in \mathcal{D}^\mu(\mathcal{F})$ and $f \in \mathcal{L}_\lambda$, then $P f \in \mathcal{L}_{\lambda - \mu}$. 

\begin{theorem} (Kashiwara) \cite{KASHIWARA}
There is a twisted differential operator $P_\mu \in \mathcal{D}^\mu(\mathcal{F})$ on $\mathcal{F}$ such that for any $\lambda, \mu$ nonnegative weights for $G$:

\begin{equation}
P_{\mu} f_G^{\lambda + \mu} = b_{\mu}(\lambda) f_G^{\lambda}
\end{equation}

where $b_\mu(\lambda) = \prod_{\alpha \in R_+} \prod_{i = 1}^{h_\alpha(\mu)} (h_\alpha(\lambda) + h_\alpha(\rho) + i)$.
\end{theorem}

\begin{proof}
Following Kashiwara, we can trivialize the sheaf of differential operators on $G$ using right translation: $\mathcal{D}(G) = \mathbb{C}[G] \otimes R(U\mathfrak{g})$, where $R: U\mathfrak{g} \rightarrow \mathcal{D}(G)$ is the extension to $U\mathfrak{g}$ of the map $R: \mathfrak{g} \rightarrow TG$ given by right translation. Then by quantum Hamiltonian reduction \cite{ETINGOFGINZBURG}, $\mathcal{D}^\mu(\mathcal{F}) = \left(\frac{\mathcal{D}(G)}{\mathcal{D}(G) \mathfrak{n}}\right)^{\mathfrak{b}_{-\mu}}$, where the subscript on $\mathfrak{b}$ denotes the subset of elements of weight $-\mu$. The quotient $\frac{\mathcal{D}(G)}{\mathcal{D}(G) \mathfrak{n}} \simeq \mathbb{C}[G] \otimes (U \mathfrak{g}/((U \mathfrak{g}) \mathfrak{n}))$. Using the well-known decomposition $\mathbb{C}[G] = \oplus (V_\nu \otimes V_\nu^*)$ as $G$-bimodules and rearranging tensor products, we get that $\mathcal{D}_\mu(\mathcal{F}) = \oplus (V_\nu \otimes (V_\nu^* \otimes \frac{U \mathfrak{g}}{(U \mathfrak{g}) \mathfrak{n}}))^{\mathfrak{b}_{-\mu}}$. Because we trivialized by right translation, and because the decomposition separates the actions of left and right translation, $\mathcal{D}^\mu(\mathcal{F}) = \oplus V_\nu \otimes (V_\nu^* \otimes \frac{U \mathfrak{g}}{(U \mathfrak{g}) \mathfrak{n}})^{\mathfrak{b}_{-\mu}}$, where $G$ acts on $V_\nu$, while 

We will use two results that Kashiwara proved; the proofs are given in the appendix. Define $U \mathfrak{h} = S\mathfrak{h}$ as the symmetric algebra and universal enveloping algebra of $\mathfrak{h}$.

\begin{theorem}
Let $V$ be a finite dimensional $\mathfrak{b}$-module. Then $(V \otimes \frac{U \mathfrak{g}}{(U \mathfrak{g}) \mathfrak{n}})^{\mathfrak{b}_{-\mu}}$ is a $S\mathfrak{h}$ module of rank equal to the dimension of the $\mu$ weight subspace of $V$. 
\end{theorem}

\begin{theorem}
Assume that the dimension of $V^\mu$ is 1. Then $(V \otimes \frac{U \mathfrak{g}}{(U \mathfrak{g}) \mathfrak{n}})^{\mathfrak{b}_{-\mu}}$ is free as a right $U \mathfrak{h}$-module.
\end{theorem}

The first theorem implies that if $\nu \ngeq \mu$, then $(V_\nu^* \otimes \frac{U \mathfrak{g}}{(U \mathfrak{g}) \mathfrak{n}})^{\mathfrak{b}_{-\mu}}$ is trivial. The second then shows that in the minimal case of $\nu = \mu$, $(V_\mu^* \otimes \frac{U \mathfrak{g}}{(U \mathfrak{g}) \mathfrak{n}})^{\mathfrak{b}_{-\mu}}$ is a free rank 1 right $U \mathfrak{h}$-module. Call the generating element $P \in (V_\mu^* \otimes \frac{U \mathfrak{g}}{(U \mathfrak{g}) \mathfrak{n}})^{\mathfrak{b}_{-\mu}}$. Then by attaching a highest weight vector $v_\mu \in V_\mu$, we get a differential operator $P_\mu = v_\mu \otimes P \in \mathcal{D}_\mu(\mathcal{F})$. Note that $P_\mu$ varies under left translation as $v_\mu$ does, that is, $g P_\mu = (g v_\mu) \otimes P$ where $g$ acts on differential operators by left translation. 

By the definition of $\mathcal{D}_\mu(\mathcal{F})$, $P_\mu f^{\lambda + \mu} \in \Gamma(\mathcal{F}, L_\lambda)$. Further, as the map of application of differential operators $\mathcal{D}_\mu(\mathcal{F}) \otimes \Gamma(\mathcal{F}, L_{\lambda + \mu}) \rightarrow \Gamma(\mathcal{F}, L_\lambda)$ is $G$-equivariant, the image will be of weight $-\lambda$. By Borel-Weil, $\Gamma(\mathcal{F}, L_\lambda) \simeq V_\lambda^*$, so there is only one element of weight $-\lambda$ (up to scaling). But $f_\lambda$ is already of weight $-\lambda$ - so we get that $P_\mu f^{\lambda + \mu} = b_\mu(\lambda) f^\lambda$. 

Kashiwara further found an explicit formula: 
$$b_\mu(\lambda) = \prod_{\alpha \in R_+} \prod_{i = 1}^{h_\alpha(\mu)} (h_\alpha(\lambda) + h_\alpha(\rho) + i)$$
and that this is the generator of the ideal of all solutions to the functional equation.  
\end{proof}

One result was not stated by Kashiwara but which will be useful for us is the following: 
\begin{lemma}
For any $\mu, \nu$, $P_\mu P_\nu = P_{\mu + \nu}$.
\end{lemma}
\begin{proof}
Consider the decomposition of $\mathbb{C}[G]$ into $\oplus V_\mu \otimes V_\mu^*$ as an increasing algebra filtration, $\mathbb{C}[G]_{\leq \mu} = \oplus_{\nu \leq \mu} V_\nu \otimes V_\nu^*$ with $\mathbb{C}[G]_{\leq \mu} \mathbb{C}[G]_{\leq \nu} \subseteq \mathbb{C}[G]_{\leq \mu + \nu}$, where $\mu \leq \nu$ is the usual partial order on weights. Then by using the trivial filtration on $U\mathfrak{g}$, we obtain an increasing filtration on $\mathcal{D}(G) \simeq \mathbb{C}[G] \otimes U\mathfrak{g}$; as the action of $\mathfrak{g}$ respects the filtration on $\mathbb{C}[G]$, this filtration is an increasing algebra filtration. This filtration then descends to $\mathcal{D}(\mathcal{F})$, so we obtain an increasing filtration $\mathcal{D}_\mu^{\leq \eta}(\mathcal{F}) = \oplus_{\nu \leq \eta} V_\nu \otimes (V_\nu^* \otimes \frac{U \mathfrak{g}}{(U \mathfrak{g}) \mathfrak{n}})^{\mathfrak{b}_{-\mu}}$.

The highest weight in $\frac{U \mathfrak{g}}{(U \mathfrak{g}) \mathfrak{n}}$ is 0, so if $\nu < \mu$, then $\mathcal{D}_\mu^{\leq \nu} = \{0\}$, and $\mathcal{D}_\mu^{\leq \mu}(\mathcal{F}) = V_\mu \otimes (V_\mu^* \otimes \frac{U \mathfrak{g}}{(U \mathfrak{g}) \mathfrak{n}})^{\mathfrak{b}_{-\mu}}$. By definition, $P_\mu \in \mathcal{D}_\mu^{\leq \mu}(\mathcal{F})$. Therefore, $P_\mu P_\nu \in \mathcal{D}_{\mu + \nu}^{\leq (\mu + \nu)}(\mathcal{F})$.

Under left translation, $P_\mu$ acts as $v_\mu$, and therefore has weight $\mu$. Therefore, $P_\mu P_\nu$ has weight $\mu + \nu$. But $P_\mu P_\nu \in \mathcal{D}_{\mu + \nu}^{\leq \mu + \nu}(\mathcal{F}) = V_{\mu + \nu} \otimes (V_{\mu + \nu}^* \otimes \frac{U \mathfrak{g}}{(U \mathfrak{g}) \mathfrak{n}})^{\mathfrak{b}_{-\mu - \nu}}$, where left translation acts on the $V_{\mu + \nu}$. As $v_{\mu + \nu} \in V_{\mu + \nu}$ is the unique element of weight $\mu + \nu$ up to scaling, $P_\mu P_\nu$ can be expressed as $v_{\mu + \nu} \otimes P'$ for some $P' \in (V_\mu^* \otimes \frac{U \mathfrak{g}}{(U \mathfrak{g}) \mathfrak{n}})^{\mathfrak{b}_{-\mu}}$; similarly, $P_{\mu + \nu} = v_{\mu + \nu} \otimes P$. Then by Kashiwara's proof that $P$ generates $(V_\mu^* \otimes \frac{U \mathfrak{g}}{(U \mathfrak{g}) \mathfrak{n}})^{\mathfrak{b}_{-\mu}}$ over $U\mathfrak{h}$, $P' = P a$ for some $a \in U\mathfrak{h}$. Using Kashiwara's calculation, we can check that $P_{\mu + \nu} f^{\lambda + \mu + \nu} = P_\mu P_\nu f^{\lambda + \mu + \nu}$ for any $\lambda$, which means that $a = 1$ and $P' = P$ - so $P_\mu P_\nu = P_{\mu + \nu} $. 
\end{proof}

\section{Relationships among $P_\mu$}
We have proven that $P_\mu P_\nu = P_{\mu + \nu}$. This situation is common enough that it is useful to study it on its own; it will turn out to be true for the two relevant families of differential operators studied in this paper. This equation on differential operators implies the following equation on b-functions (with subscripts omitted because it works for either subscript):

\begin{equation} \label{Cocycle Equation}
b_\mu(\lambda) b_\nu(\lambda + \mu) = b_{\mu + \nu}(\lambda)
\end{equation}

which can be seen as a 1-cocycle equation in the bar resolution complex for the group cohomology of $\Lambda$, where the action of $\Lambda$ is on the multiplicative group $\mathbb{C}(\mathbb{C}\Lambda)^\times$ of rational functions (not including $0$) on $\mathbb{C}\Lambda$ by translation, $T_\mu(p)(\lambda) = p(\lambda + \mu)$.

\begin{lemma} \label{Hyperplane decomposition}
Let $\{b_\mu\}$ be a cocycle (i.e. satisfy the above equation) such that for $\mu \in \Omega$, $b_\mu$ is a polynomial (not just a rational function). Then for any $\mu \in \Omega$, $b_\mu$ can be expressed as $\prod_i (\alpha_i(\lambda) + k_i)$ with $\alpha_i \in Hom(\Lambda, \mathbb{Z}), k_i \in \mathbb{C}$. 
\end{lemma}
Similar results have been found for b-functions of multiple functions; see Sabbah, theorem 4.2.1 \cite{SABBAH} and Gyoja \cite{GYOJA}.

\begin{proof}
From the cocycle equation, we can switch $\mu, \nu$ to get that $b_\mu T_\mu(b_\nu) = b_\nu T_\nu(b_\mu)$. 

Assume $p|b_\mu$ for some irreducible $p$; we need to prove that $p$ is of the form $\alpha(\lambda) + k$ for some $\alpha \in Hom(\Lambda, \mathbb{Z}), k \in \mathbb{C}$. We first check that for any $\nu \in \Omega$, either there is a nonnegative integer $a$ such that $p|T_{\nu - a \mu} b_\mu$ or $p$ is invariant under translation by $\mu$.

By the second equation, we know that $p | b_\nu T_\nu b_\mu$. As polynomials over $\mathbb{C}\Lambda$ form a UFD and $p$ is irreducible, we know that either $p|b_\nu$ or $p|T_\nu b_\mu$. In the latter case, we are done, so assume we are in the former case. Then by translation, we also have that $T_\mu p|T_\mu b_\nu|b_\mu T_\mu b_\nu = b_\nu T_\nu b_\mu$. By repeating this, we either have infinitely many $a$ such that $p(\lambda + a\mu)|b_\nu(\lambda)$, or for some $a$ that $T_{a \mu} p|T_\nu b_\mu$. As before, assume we are in the former case. Then as $b_\nu(\lambda)$ has finitely many factors, there can only be finitely many distinct $T_{a \mu} p$; therefore they must be equal for some $a_0, a_1$, and therefore, $T_{(a_1 - a_0) \mu} p = p$. This means that for any $\lambda$, the function $t \rightarrow p(\lambda + t \mu)$ is periodic. But $p$ is a polynomial, so the function is a polynomial, which can only be periodic if it is constant, so $p(\lambda + t \mu) = p(\lambda)$ for any $t$, and $p$ is invariant under translation by any multiple of $\mu$. 

Assume we are in the first case, that is, that $T_{a \mu} p | T_\nu b_\mu$. Then we know that $T_{a \mu - \nu} p|b_\mu$. As the former is a translation of $p$, it is also irreducible - so we can repeat the above process. As $b_\mu$ has finitely many factors, there can only be finitely many distinct such translations; therefore, for some $m_0 < m_1, n_0 < n_1$ we get that $T_{m_1 \mu - n_1 \nu} p = T_{m_0 \mu - n_0 \nu} p$. Therefore there are integers $m, n$ with $n > 0$ such that $T_{m \mu - n \nu} p = p$. Dropping the assumption that $T_{a \mu} p | T_\nu b_\mu$ and going back to the general case (i.e. allowing $T_{a \mu} p = p$ as in the above paragraph), we get that there are integers $m, n$ nonnegative and not both $0$ such that $T_{m \mu - n \nu} p = p$. 

Choose generators $\beta_i$ of $\Lambda$; then as $p$ is a nonconstant polynomial, for some $i_0$ we get that $T_{\beta_{i_0}} p \neq p$. Then for all $i \neq i_0$, we know that there are $m_i, n_i$ such that $p(\lambda + m_i \beta_{i_0} - n_i \beta_i) = p(\lambda)$. Hence $n_i \neq 0$. For any ${\mu \in span(\{m_i \beta_{i_0} - n_i \beta_i\})}$, we have that $p(\lambda + \mu) = p(\lambda)$. This is a subspace of codimension 1 in $\mathbb{C}\Lambda$, so there is some $\alpha \in Hom(\Lambda, \mathbb{C})$ such that $\alpha(\mu) = 0$ if and only if $p(\lambda + \mu) = p(\lambda)$ for all $\lambda$; further, it's clear that $\alpha \in Hom(\Lambda, \mathbb{Z})$. We then have that for any $\mu$ such that $\langle \alpha, \mu \rangle = 0$, $p(\lambda + \mu) = p(\lambda)$. Choose a point $\lambda_0$ such that $p(\lambda_0) = 0$. Then for any point where $\alpha(\lambda) - \alpha(\lambda_0) = 0$, we have that $p(\lambda) = p(\lambda + (\lambda_0 - \lambda)) = 0$. But $\alpha(\lambda) - \alpha(\lambda_0)$ is an irreducible polynomial, so $\alpha(\lambda) - \alpha(\lambda_0)$ divides $p$. Since $p$ is irreducible, $p = \alpha(\lambda) - \alpha(\lambda_0)$. Setting $k = -\alpha(\lambda_0)$, we have proven the lemma.

\end{proof}

\begin{remark}
A useful way to think about this result is to consider the bar resolution complex in which the b-function is a 1-cocycle. A generalization of this proof can be used to give the first cohomology group of the complex. We can then provide a "lift" $A(\lambda)$ for each cocycle in a canonical way as a product of gamma functions of rational hyperplanes, such that $b_\mu(\lambda) = \frac{A(\lambda + \mu)}{A(\lambda)}, A(0) = 1$. For any $\mu \in \Lambda$, we then have that $b_\mu(0) = \frac{A(\mu)}{A(0)} = A(\mu)$. The following corollaries are then natural. 
\end{remark}

Consider $\alpha \in Hom(\Lambda, \mathbb{Z})$ positive and indivisible (that is, there is no nontrivial $a \in \mathbb{Z}$ with $\frac{\alpha}{a} \in Hom(\Lambda, \mathbb{Z})$, and for any $\mu \geq 0$, $\langle \alpha, \mu \rangle \geq 0$). For each $\mu \in \Lambda^+$, let $K_{\alpha, \mu} = \{k s.t. (\alpha(\mu) + k)|b_\mu\}$ with multiplicity. Then for any $\mu, \nu \in \Lambda^+$

\begin{equation}
K_{\alpha, \mu} \cup (K_{\alpha, \nu} + \langle \alpha, \mu\rangle) = K_{\alpha, \nu} \cup (K_{\alpha, \mu} + \langle \alpha, \nu\rangle)
\end{equation}

where $A + k = \{a + k|a \in A\}$; this follows directly from unique factorization and the cocycle equation. 

\begin{corollary} \label{Source Corollary}
Let $\alpha \neq 0 \in Hom(\Lambda, \mathbb{Z})$. Then there is some $c \in \mathbb{Z}_{\geq 0}$ such that $|K_{\alpha, \mu}| = c \langle \alpha, \mu \rangle$.
\end{corollary}
\begin{proof}
Let $\nu \in \Lambda^+, \langle \alpha, \nu\rangle \neq 0$, and let $c = \frac{|K_{\alpha, \nu}|}{\langle \alpha, \nu} \rangle$. Then for any $\mu \in \Lambda^+$, by summing the $K$ equation over the sets with multplicity, we get that:
\begin{equation}
\sum K_{\alpha, \mu} + \sum K_{\alpha, \nu} + |K_{\alpha, \nu}| \langle \alpha, \mu\rangle = \sum K_{\alpha, \nu} + \sum K_{\alpha, \mu} + |K_{\alpha, \mu}| \langle \alpha, \nu\rangle. 
\end{equation}
We can subtract and divide to get $c = \frac{|K_{\alpha, \nu}|}{\langle \alpha, \nu\rangle}$; then as $\alpha$ is indivisible and positive, $c$ must be a nonnegative integer.
\end{proof}

\begin{corollary} \label{Coefficient Zero}
For any $\alpha \neq 0 \in Hom(\Lambda, \mathbb{Z}), \mu \in \Lambda$ with $\langle \alpha, \mu\rangle = 0$, $K_{\alpha, \mu} = \emptyset$. 
\end{corollary}

\begin{corollary} \label{Coefficients Equal}
Let $\mu, \nu \in \Lambda^+$, and let $\alpha \in Hom(\Lambda, \mathbb{Z})$ such that $\alpha(\mu) = \alpha(\nu) = a$. Then $K_{\alpha, \mu} = K_{\alpha, \nu}$.
\end{corollary}
\begin{proof}
Assume otherwise. If we allow 

$$(K_{\alpha, \mu} + a) - K_{\alpha, \mu} = (K_{\alpha, \nu} + a) - K_{\alpha, \nu}$$

But the map $S \rightarrow (S + a) - S$ is a $\mathbb{Z}$-linear map where every element of the kernel has infinitely many elements, and both $K_{\alpha, \mu}$ and $K_{\alpha, \nu}$ are finite - so they must be the same.
\end{proof}

\section{The differential operator for $f_G$}
We now want to find the differential operators $P_{\mu, G^3}$ that give us the functional equation 
\begin{equation}
P_{\mu, G^3} f_{G^3}^{\lambda + \mu} = b_{\mu, G^3}(\lambda) f_{G^3}^\lambda
\end{equation}

In order to do this, we look first for the b-function for $f_K$, the $B_1 \times B_2 \times P$-semiinvariant section, using Kashiwara's calculation, and then figure out the relationship between $b_{\mu, G^3}$ and $b_{\mu, K}$. As $f_K = f_G \otimes f_G \otimes f_{\mathbb{P}}$ is a product of three functions on the three factors of $X$, we can set $P_{\mu, K} = P_{G, \mu_1} \otimes P_{G, \mu_2} \otimes P_{\mathbb{P}, m}$ and get the functional equation
\begin{multline*}
P_{\mu, K} f_K^{\lambda + \mu} = (P_{G, \mu_1} f_G^{\lambda_1 + \mu_1}) \otimes (P_{G, \mu_2} f_G^{\lambda_2 + \mu_2}) \otimes (P_{\mathbb{P}, m} f_{\mathbb{P}}^{l + m})\\
= b_{G, \mu_1}(\lambda_1) b_{G, \mu_2}(\lambda_2) b_{\mathbb{P}, m}(l) f_K^\lambda
\end{multline*}
where $b_{\mathbb{P}, m}$ is the b-function for a hyperplane in $\mathbb{P}^{n - 1}$. It is easy to see that $b_{\mathbb{P}, m}(l) = \prod_{i = 1}^m (l + i)$; combining that with the calculation in section 5, we get that:

\begin{equation}
P_{\mu, K} f_K^{\lambda + \mu} = b_{\mu, K}(\lambda) f_K^{\lambda}
\end{equation}
where 
$$b_{\mu, K} = \prod_{\alpha \in R_+} \prod_{i = 1}^{h_\alpha{\mu_1}} (h_\alpha(\lambda_1) + h_\alpha(\rho) + i) \prod_{j = 1}^{h_\alpha{\mu_2}} (h_\alpha(\lambda_2) + h_\alpha(\rho) + j) \prod_{k = 1}^m (l + i)$$

As we now have $b_{\mu, K}$, we only need to determine the relationship between the two b-functions. 

\begin{lemma}
Define $H(\mu) = \langle u_{-\mu}, u_\mu\rangle$. Then $u_\mu = H(\mu) v_\mu^G, u_{-\mu} = H(\mu) v_{-\mu}^G$.
\end{lemma}
\begin{proof}
Because $u_\mu$ is the unique $G$-invariant up to scaling, $u_\mu = c v_\mu^G$ for some $c$. Then $H(\mu) = \langle u_{-\mu}, u_\mu\rangle = \langle u_{-\mu}, c v_\mu^G\rangle = c \langle u_{-\mu}, v_\mu \rangle = c$, so $u_\mu = H(\mu) v_\mu^G$. Similarly, $u_{-\mu} = H(\mu) v_{-\mu}^G$. 
\end{proof}

Let $P_{\mu, G} = u_\mu \otimes P \in \mathcal{D}_\mu(X)$. Then as $P_{\mu, K} = v_\mu \otimes P$, the last lemma shows that $P_{\mu, G^3} = H(\mu) P_{\mu, K}^G$. 

\begin{lemma}
$P_{\mu, G^3} f_{G^3}^{\lambda + \mu} = b_{\mu, G^3}(\lambda) f_{G^3}^\lambda$ where $b_{\mu, G^3}(\lambda) = \frac{H(\lambda + \mu)}{H(\lambda)} b_{\mu, K}$
\end{lemma}
\begin{proof}
Consider differential operators of the form $V_\mu \otimes {P}$, and the restriction of the map of application of differential operators $\mathcal{D}_\mu(X) \otimes \Gamma(X, L_{\lambda + \mu}) \rightarrow \Gamma(X, L_\lambda)$ to only include such differential operators. Then this is isomorphic to a $G \times G \times G$-equivariant map $V_\mu \otimes V_{\lambda + \mu}^* \rightarrow V_\lambda^*$, where the isomorphism takes $P_{\mu, K} \rightarrow v_\mu, P_{\mu, G} \rightarrow u_\mu, f_K^{\lambda} \rightarrow v_{-\lambda}, f_{G}^\lambda \rightarrow u_{-\lambda}$. But up to scaling, there is only one such map. Under this interpretation, we have $v_\mu \otimes v_{-\lambda - \mu} \rightarrow b_{K, \mu}(\lambda) v_{-\lambda}$ and $u_\mu \otimes u_{-\lambda - \mu} \rightarrow b_{G, \mu}(\lambda) u_{-\lambda}$. By pairing with the respective vectors and considering the fact that this map is $G \times G \times G$-invariant, we can think of this as a map $V_\mu \otimes V_\lambda \rightarrow V_{\mu + \lambda}$ that takes $v_\mu \otimes v_\lambda \rightarrow v_{\mu + \lambda}, H(\lambda + \mu) u_\mu \otimes u_\lambda \rightarrow H(\lambda) b_{G, \mu}(\lambda) u_{\lambda + \mu}$. 

But by Borel-Weil, we also have an isomorphism between $\Gamma(X', L_{-\lambda})$ and $V_\lambda$; as such, we have the multiplication map on sections. The multiplication map takes $v_\mu \otimes v_\lambda$ to $v_{\mu + \lambda}$ - so the map we want is the multiplication map scaled by $b_{K, \mu}$. The multiplication map also takes $u_\mu \otimes u_\lambda$ to $u_{\mu + \lambda}$. Therefore, the map we want takes $H(\lambda + \mu) u_\mu \otimes u_\lambda$ to $H(\lambda + \mu) u_{\lambda + \mu}$. So $b_{\mu, G^3}(\lambda) = \frac{H(\lambda + \mu)}{H(\lambda)} b_{\mu, K}$. 
\end{proof}
\begin{remark}
This lemma is easier to understand in terms of the lift of the cocycle, as explained in the remark in the previous section. If $A_K(\lambda)$ is the lift of $b_{\mu, K}$ and $A_{G^3}(\lambda)$ is the lift of $b_{\mu, G^3}$, then for $\lambda \in \Omega$, $A_{G^3}(\lambda) = H(\lambda) A_K(\lambda)$. 
\end{remark}
Similar reasoning applied to the map $V_\mu \rightarrow \mathcal{D}_\mu(X), v \rightarrow v \otimes P$ and the differential operator multiplication map also implies that $P_{\mu, G^3} P_{\nu, G^3} = P_{\mu + \nu, G^3}$. As such, all of the conclusions from the previous section apply to the b-function we are trying to find.

\section{The functions $H(\lambda)$}
We now wish to find $H(\lambda)$. The value of $H(\lambda)$ depends on the choice of the 6 subgroups from before; for certain values of $\lambda$, we can choose convenient subgroups to make the calculation easier. As such, we first need to show how the choice of subgroup changes $H(\lambda)$. 

In the last section, we defined $H(\lambda) = \langle u_{-\lambda}, u_\lambda \rangle$ where $u_\lambda$ and $u_{-\lambda}$ were normalized with $\langle u_{-\lambda}, v_\lambda \rangle = \langle v_{-\lambda}, u_\lambda\rangle = \langle v_{-\lambda}, v_\lambda\rangle = 1$; however, $H(\lambda)$ is easier to understand in an "un-normalized" form, $H(\lambda) = \frac{\langle u_{-\lambda}, u_\lambda\rangle \langle v_{-\lambda}, v_\lambda\rangle}{\langle u_{-\lambda}, v_\lambda \rangle \langle v_{-\lambda}, u_\lambda \rangle}$. In this form, we can instead choose arbitrary $G$-invariant $u_\lambda, u_{-\lambda}$, $B_1 \times B_2 \times P$-semiinvariant $v_\lambda$, and $B_1' \times B_2' \times P'$-semiinvariant $v_{-\lambda}$. We can then regard $H(\lambda)$ as a meromorphic function $H_\lambda(x \times x')$ on $X \times X'$. 

\begin{lemma}
The function $\frac{H_{\lambda + \mu}}{H_\lambda H_\mu}$ is constant on $X \times X'$.
\end{lemma}
\begin{proof}
Any point $x_0 \times x'_0$ has some $g = (g_1, g_2, g_3, g_4, g_5, g_6)$ such that $x_0 \times x'_0 = (g_1, g_2, g_3, g_4, g_5, g_6) x \times x'$. Then as $(g_1, g_2, g_3) v_\lambda$ is a highest weight element with respect to the subgroup given by $x_0$ (and a similar statement for $v_{-\lambda}$, 

\begin{equation}H_\lambda(x_0 \times x'_0) = \frac{\langle u_{-\lambda}, u_\lambda \rangle \langle (g_4, g_5, g_6) v_{-\lambda}, (g_1, g_2, g_3) v_\lambda \rangle}{\langle u_{-\lambda}, (g_1, g_2, g_3) v_\lambda \rangle \langle (g_4, g_5, g_6) v_{-\lambda}, u_\lambda \rangle}
\end{equation}

But using the correspondence $v_\lambda \rightarrow f_K^\lambda, u_\lambda \rightarrow f_{G}^\lambda$, it's clear that each of the three products other than $\langle u_\lambda, u_{-\lambda} \rangle$ cancels; we therefore get that for any $x_0 \times x'_0$, we have
$$\frac{H_{\lambda +
    \mu}}{H_\lambda H_\mu} = \frac{\langle u_{-\mu - \lambda}, u_{\mu +
    \lambda} \rangle}{\langle u_{-\mu}, u_\mu \rangle \langle
  u_{-\lambda}, u_\lambda \rangle}
$$ Therefore, it is constant. 
\end{proof}

This lemma shows that when we change our choice of subgroups, it changes $\frac{H(\lambda + \mu)}{H(\lambda) H(\mu)}$ by a constant independent of $\lambda$ (and exponential in $\mu$). As we only care about $b_\mu(\lambda)$ up to a constant (depending on $\mu$, but not $\lambda$), this means that we can choose any subgroup to find the $H(\lambda)$. 

It will turn out that we only need to know $H$ on certain subcones of $\Omega$. Ginzburg and Finkelberg found the generators of $\Omega$ in the proof of lemma 5.5.1 \cite{GINZBURGFINKELBERG}; there are two families of generators. Label them as follows:

\[
\alpha_i = (\wedge^i \omega, \wedge^{n - i} \omega, 0), 1 \leq i \leq {n - 1}
\]

\[
\beta_j = (\wedge^{j - 1} \omega, \wedge^{n - j} \omega, 1), 1 \leq j \leq n
\]

We then only need find $H$ on the following subcones:

\[
\Delta = span(\{\alpha_i\}_{i = 1}^n)
\]

\[
\Delta_{< j} = span({\beta_j} \cup \{\alpha_i\}_{i < j})
\]

\[
\Delta_{\geq j} = span({\beta_j} \cup \{\alpha_i\}_{i \geq j})
\]

For each $j$, choose $B_1, B_2'$ as the upper triangular subgroup, $B_1', B_2$ as the lower triangular subgroup, $P$ fixing $e_j$, and $P'$ fixing $e'_j$. Let $R_+$ be the set of positive roots of $G = SL_n$. For each $j$, define the following function on positive roots:

For $\gamma \in R_+$, define 

\begin{equation}
\chi_j(\gamma) = \begin{cases}
       1 & \text{if} \enspace \gamma > \wedge^j \omega \\
       1 & \text{if} \enspace \gamma > \wedge^{j - 1} \omega\\
       0 & \quad \text{else}
\end{cases}
\end{equation}

\begin{theorem}
If $\lambda \in \Delta$, then $H(\lambda) = dim(V_{\lambda_1})$. 
If $\lambda \in \Delta_{i \geq j}$ or $\lambda \in \Delta_{i < j}$ with $\lambda = (\sum_i r_i \alpha_i) + s_j \beta_j$, then 
$$H(\lambda) =  \frac{\prod_{\gamma \in R_+} (h_\gamma(\rho + \sum r_i \alpha_i) + s_j \chi_j(\gamma))}{\prod_{\gamma \in R_+} h_\gamma(\rho)}$$
\end{theorem}
\begin{proof}

If $\lambda \in \Delta$, then $\lambda = (\lambda_1, \lambda_1^*, 0)$, where $V_{\lambda_1^*} = V_{\lambda_1}^*$. Then by choosing $u_\lambda \in V_\lambda$ corresponding to the identity map on $V_{\lambda_1}$, $u_{-\lambda \in V_{\lambda}^*}$ corresponding to the trace map, $v_\lambda = v_{\lambda_1} \otimes v_{-\lambda_1}$ and $v_{-\lambda} = v_{-\lambda_1} \otimes v_{\lambda_1}$, we can see that:

\begin{equation} \label{HonDelta}
H(\lambda) = dim(V_{\lambda_1}) = \prod_{\gamma \in R_+} \frac{h_\gamma(\lambda_1 + \rho)}{h_\gamma(\rho)}
\end{equation}

If $\lambda \in \Delta_{\geq j}$, then let 
\begin{align*}
\lambda &= (\lambda_1, \lambda_2, l) = s_j \beta_j + \sum_{i \geq j} r_i \alpha_i\\
\lambda' &= (\lambda_1', \lambda_2', l') = s_j \alpha_j + \sum_{i \geq j} r_i \alpha_i\\
\lambda^0 &= (\lambda_1^0, \lambda_2^0, l^0) = \sum_{i \geq j} r_i \alpha_i
\end{align*}

We now define several maps to be used later. Let us denote the wedge map $\wedge_j: V_{\wedge^{j - 1} \omega} \otimes \mathbb{C}^n \rightarrow V_{\wedge^j \omega}$; note that it is $G$-equivariant. By copying this map $s_j$ times, we get a map $\wedge_j^{s_j}: V_{s_j \wedge^{j - 1} \omega} \otimes Sym^{s_j} \mathbb{C}^n \rightarrow V_{s_j \wedge^j \omega}$ taking $v_{s_j \wedge^{j - 1} \omega} \otimes v_{s_j \omega} \rightarrow v_{s_j \wedge^j \omega}$. Then by tensoring with $Id_{V_{\lambda^0}}$, we have a map $f: (V_{s_j \wedge^{j - 1} \omega} \otimes V_{\lambda_1^0}) \otimes V_{\lambda_2^0} \otimes V_{s_j \omega} \rightarrow (V_{s_j \wedge^j \omega} \otimes V_{\lambda_1^0}) \otimes V_{\lambda_2^0}$. Then let $g: V_\lambda \rightarrow V_{\lambda'}$ be defined by the composition:

\begin{equation}
V_\lambda \xrightarrow{i_{s_j \wedge^{j - 1} \omega, \lambda_1^0}} (V_{s_j \wedge^{j - 1} \omega} \otimes V_{\lambda_1^0}) \otimes V_{\lambda_2^0} \otimes V_{s_j \omega} \xrightarrow{f} (V_{s_j \wedge^j \omega} \otimes V_{\lambda_1^0}) \otimes V_{\lambda_2^0} \xrightarrow{p_{s_j \wedge^j \omega, \lambda_1^0}} V_\lambda'
\end{equation}
where $i$ is the natural inclusion map that takes $v_\lambda$ to $v_{s_j \wedge^{j - 1} \omega} \otimes v_{\lambda_1^0}$ and $p$ is the natural quotient map that takes $v_{s_j \wedge^j \omega} \otimes v_{\lambda_1^0}$ to $v_\lambda'$. Then $g$ is $G$-equivariant and $g(v_\lambda) = v_\lambda'$. 

Because $g$ is $G$-equivariant, we have that $g(u_\lambda) = g(H(\lambda) v_\lambda^G) = H(\lambda) g(v_\lambda)^G = H(\lambda) v_{\lambda'}^G = \frac{H(\lambda)}{H(\lambda')} u_{\lambda'}$. Therefore, if we define $c$ by $g(u_\lambda) = c u_{\lambda'}$, then $H(\lambda) = c H(\lambda')$; as $\lambda' \in \Delta$, we know $H(\lambda')$, so we only need to find $c$.

Let $g^*: V_{\lambda'}^* \rightarrow V_\lambda^*$ be the transpose of $g$. By the definition of $g$, we have 
$$g^*(v_{-\lambda'}) = g^*(v_{-\lambda_1'} \otimes v_{-\lambda_2}) = (p \circ f)^*((v_{-\lambda_1^0} \otimes v_{- s_j \wedge^j \omega}) \otimes v_{-\lambda_2})$$

Using the definition of $f$, we want to know $(\wedge_j^{s_j})^* v_{- s_j \wedge^j \omega}$, where $(\wedge_j^{s_j})^*$ is the transpose of $\wedge_j^{s_j}$. Let $SL_j \subset G$ act on the first $j$ coordinates in both $\mathbb{C}^n$ and its dual. By restricting our focus to the subspaces where $SL_j$ acts nontrivially (and doing the same for the dual spaces), we obtain a $SL_j$-equivariant map $\wedge_j^{s_j}: V_{s_j \wedge^{j - 1} \omega_j} \otimes V_{s_j \omega_j} \rightarrow V_{s_j \wedge^j \omega_j} = \mathbb{C}$ taking $v_{s_j \wedge^{j - 1} \omega_j} \otimes v_{s_j \omega_j} \rightarrow 1$; in other words, a $SL_j$-invariant element $w \in V_{s_j \wedge^{j - 1} \omega_j}^* \otimes V_{s_j \omega_j}^*$ such that $\langle w, v_{s_j \wedge^{j - 1} \omega_j} \otimes v_{s_j \omega_j} \rangle = 1$.

But this corresponds to the conditions we've already studied if we take $G' =\nobreak SL_j$, ${\delta = (s_j \wedge^{j - 1} \omega_j, s_j \omega_j, 0) \in \Delta_{G'}}$; by definition, $w = u_{- \delta}$. Then by the uniqueness up to scalar of the $G'$-invariant, $w = c_1 v_{- \delta}^{G'}$ for some $c_1$. By pairing with $v_\delta$, we get that $1 = \langle w, v_\delta\rangle = \langle c_1 v_{-\delta}^{G'}, v_\delta\rangle = c_1 \langle v_{-\delta}, v_\delta^{G'}\rangle$. Then as $u_\delta = H_{SL_j}(\lambda) v_\delta^{G'}$, we get that $H_{SL_j}(\delta) = c_1 \langle v_{-\delta}, u_\delta\rangle = c_1$. 

We therefore have that $w = H_{SL_j}(\delta) v_{-\delta}^{G'}$. By expanding our focus again to all $n$ coordinates, we have that $\wedge_j^{s_j*} v_{- s_j \wedge^j \omega} = H_{SL_j}(\delta) (v_{- s_j \wedge^{j - 1} \omega} \otimes v_{- s_j \omega})^{SL_j}$. Therefore, as $v_{-\lambda_1^0}$ and $v_{-\lambda_2}$ are $SL_j$-invariant, $g^*(v_{-\lambda'}) = H_{SL_j}(\delta) v_{-\lambda}^{SL_j}$. Then by pairing with $u_\lambda$, we get that:

\begin{align*}
c &= \langle v_{-\lambda'}, c u_{\lambda'}\rangle = \langle v_{-\lambda'}, g(u_\lambda)\rangle = \langle g^*(v_{-\lambda'}), u_\lambda\rangle \\
&= \langle H_{SL_j}(\delta) v_{-\lambda}^{SL_j}, u_\lambda\rangle = H_{SL_j}(\delta) \langle v_{-\lambda}, u_\lambda \rangle = H_{SL_j}(\delta)
\end{align*}

where we use the $G$-invariance of $u$ to say that $\langle v_{-\lambda}^{SL_j}, u_\lambda \rangle = \langle v_{-\lambda}, u_\lambda \rangle$.

We therefore have that $H(\lambda) = H_{SL_j}(\delta) H(\lambda')$. By the calculation of $H(\lambda)$ on $\Delta$ (equation \ref{HonDelta}), we get that:

\begin{align*}
H(\lambda') &= \frac{\prod_{\gamma \in R_+} (h_\gamma(\rho + (\sum_i r_i \wedge^i \omega) + s_j \wedge^j \omega))}{\prod_{\gamma \in R_+}(h_\gamma(\rho))} \\
H_{SL_j}(\delta) &= \frac{\prod_{\gamma \in R_{SL_j +}} (h_\gamma(\rho_{SL_j} + s_j \wedge^{j - 1} \omega_j))}{\prod_{\gamma \in R_{SL_j + }} h_\gamma(\rho_{SL_j})}
\end{align*}

By considering $R_{SL_j +}$ as a subset of $R_+$, we get the second half of the theorem.
\end{proof}

For $\gamma \in R_+$, define 

\begin{equation}
\chi_j'(\gamma) = \begin{cases}
       1 & \text{if} \enspace \gamma > \wedge^j \omega + \wedge^{j - 1} \omega\\
       0 & \quad \text{else}
     \end{cases}
\end{equation}

Combining this calculation of $H(\lambda)$ and $b_{K, \mu}(\lambda)$ from earlier, we get:

\begin{corollary}
We can find $b_{G^3, \mu}(\lambda)$ if there is a subcone listed above that both $\mu, \lambda$ are in. Specifically:

If $\mu = \alpha_j, \lambda \in \Delta$ (so $\lambda = (\lambda_1, \lambda_1^*, 0)$), then

$$
b_\mu(\lambda) = \frac{\prod_{\gamma > \wedge^j \omega} (h_\gamma(\rho + \lambda_1) + 1) (h_\gamma(\rho + \lambda))}{\prod_{\gamma > \wedge^j \omega} (h_\gamma(\rho) + 1) (h_\gamma(\rho))}
$$

If $\mu = \beta_j, \lambda \in \Delta_{i < j}$ or $\lambda \in \Delta_{i \geq j}$ with $\lambda = (\sum_i r_i \alpha_i) + s_j \beta_j$, then

$$b_\mu(\lambda) = \prod_{\substack{\gamma > \wedge^{j - 1} \omega\\
\operatorname{and}\  \gamma > \wedge^j \omega}} \left(h_\gamma(\rho + \sum r_i
\alpha_i) + s_j \right) \prod_{\substack{\gamma > \wedge^{j - 1}
  \omega\\
\operatorname{or}\ \gamma > \wedge^j \omega}}\left(h_\gamma(\rho + \sum r_i \alpha_i) + s_j + 1\right)
$$

If $\mu = \alpha_k, k < j, \lambda \in \Delta_{i < j}$ or $k \geq j, \lambda \in \Delta_{i \geq j}$, then 

$$
b_\mu(\lambda) = \prod_{\gamma > \wedge^k \omega} \left(h_\gamma(\rho + \sum r_i \alpha_i) + s_j \chi_j'(\gamma) \right) \left( h_\gamma(\rho + \sum r_i \alpha_i) + s_j \chi_j(\gamma) + 1 \right)
$$
\end{corollary}

\section{The b-function}
We have now assembled enough rules to fully determine $b_{G^3, \mu}(\lambda)$ for any $\mu \in \Omega, \lambda \in \mathbb{C}\Omega$. By equation \ref{Cocycle Equation}, it is enough to find $b_{G^3, \mu}$ for $\mu$ generators of $\Omega$. Write $\lambda = \sum a_i \alpha_i + \sum b_j \beta_j$. We therefore want to find a set of polynomials $\{b_{\alpha_i}\}_{1 \leq i \leq n - 1}, \{b_{\beta_j}\}_{0 \leq j \leq n - 1}$ satisfying the following properties:
\begin{enumerate}
\item Each of the polynomials factors as a product of hyperplanes with integer coefficients, as in lemma \ref{Hyperplane decomposition}. 

\item If $a_i = -1$, $f^{\lambda + \alpha_i}$ is a $G$-invariant global section of a line bundle which has no nontrivial $G$-invariant global sections, so the functional equation implies that $b_{\alpha_i}(\lambda) = 0$, and therefore that $(a_i + 1)|b_{\alpha_i}$; this is analogous to the fact that $s + 1|b(s)$. Similarly, $(b_j + 1)|b_{\beta_j}$.

\item If $(\delta + k)|b_{\alpha_i}$ for $\delta \in Hom(\Lambda, \mathbb{Z})$, then $\langle c\delta, \alpha_i\rangle \neq 0$, and similarly for $b_{\beta_j}$, as in corollary \ref{Coefficient Zero}.

\item If $(\delta + k)|b_{\alpha_i}$ for $\delta \in Hom(\Lambda, \mathbb{Z})$, and $\langle \delta, \alpha_{i'} \rangle \neq 0$, then for some $k'$, $(\delta + k')|b_{\alpha_i'}$, and similarly for $b_{\beta_j}$, as in corollary \ref{Source Corollary}.

\item If $(\delta + k)|b_{\alpha_i}$ for $\delta \in Hom(\Lambda, \mathbb{Z})$ and $\langle \delta, \alpha_i\rangle = \langle \delta, \alpha_j\rangle$, then $(\delta + k)|b_{\alpha_j}$, and similar results with $b_{\beta_j}$, as in corollary \ref{Coefficients Equal}. 

\item If $b_{j} = 0$ for all $j$, we have
\[b_{\alpha_i} = \frac{\prod_{\gamma > \wedge^j \omega} (h_\gamma(\rho +
  \lambda_1) + 1) (h_\gamma(\rho + \lambda))}{\prod_{\gamma > \wedge^j
    \omega} (h_\gamma(\rho) + 1) (h_\gamma(\rho))}.\]
 If $s_j = 0$ for all $j \neq k$ and either $a_i = 0$ for either all
 $i < k$ or all $i \geq k$, we have 
\begin{align*}b_{\alpha_i} &= \prod_{\gamma > \wedge^k \omega}
  \left(h_\gamma(\rho + \sum r_i \alpha_i) + s_j \chi_j'(\gamma) \right)
  \left( h_\gamma(\rho + \sum r_i \alpha_i) + s_j \chi_j(\gamma) + 1
  \right)\\
b_{\beta_k} &= \prod_{\substack{\gamma > \wedge^{j - 1} \omega\\
\operatorname{and}\ \gamma >
  \wedge^j \omega}}\ \left(h_\gamma(\rho + \sum r_i \alpha_i) + s_j
\right) \prod_{\substack{\gamma > \wedge^{j - 1} \omega\\
\operatorname{or}\ \gamma > \wedge^j
  \omega}}\left(h_\gamma(\rho + \sum r_i \alpha_i) + s_j + 1\right).
\end{align*} 
\end{enumerate}

These will determine the b-functions entirely. Let $\Gamma$ be the gamma function. 

\begin{theorem}
Write $\lambda = (\sum_i a_i \alpha_i) + (\sum_j b_j \beta_j)$. Then let 

\begin{multline*}
A(\lambda) = \prod_j \Gamma (b_j + 1 ) \cdot
\prod_{\gamma \in R_+} \Gamma \left(h_\gamma(\rho) + \sum_{\omega^i <
    \gamma} a_i + \sum_{\substack{\omega^{j - 1}<\gamma \\
\operatorname{and}\ 
\omega^j < \gamma}}
b_j \right)\\
\cdot\Gamma\left(h_\gamma(\rho) + \sum_{\omega^j < \gamma} a_i +
\sum_{\substack{\omega^{j - 1}<\gamma \\
\operatorname{or}\ 
\omega^j < \gamma}}
 b_j + 1\right) 
\end{multline*}

Then $b_\mu(\lambda) = \frac{A(\lambda + \mu)}{A(\lambda)A(\mu)}$. 
\end{theorem}

\begin{proof}

Let $\alpha_i = (\wedge^i \omega, \wedge^{n - i} \omega, 0) \in \Delta$. By property 1, $b_{\alpha_i}(\lambda) = \prod_k \langle \eta_k, \lambda\rangle + m_k$ for some $\eta_k, m_k$ such that $\langle \eta_k, \alpha_i\rangle \neq 0$. By property 6, we have that $b_{\alpha_i}(\lambda) = \frac{\prod_{\gamma \in R_+} (h_\gamma(\rho + \lambda_1) + 1) h_\gamma(\rho + \lambda)}{\prod_{\gamma \in R_+} (h_\gamma(\rho) + 1) h_\gamma(\rho)}$ if $\lambda \in \mathbb{C}\Delta$. By combining these, we get that we can partition the factors of $b_{\alpha_i}$ into two families of hyperplanes. The families are both indexed by the roots $\gamma \in R_+$ such that $\gamma > \wedge^i \omega$; define $L_{\gamma, \alpha_i}$ as the hyperplane that, on $\Delta$, is equal to $h_\gamma(\lambda_1) + h_\gamma(\rho) + 1$, and $L'_{\gamma, \alpha_i}$ as the hyperplane that, on $\Delta$, is equal to $h_\gamma(\lambda_1) + h_\gamma(\rho)$. 

By property 5, if $\gamma > \wedge^i \omega, \wedge^{i'} \omega$, then $L_{\gamma, \alpha_i}|b_{\alpha_{i'}}$; the only hyperplane it could match is $L_{\gamma, \alpha_{i'}}$, so $L_{\gamma, \alpha_i} = L_{\gamma, \alpha_{i'}}$. Similarly, $L_{\gamma, \alpha_i}' = L_{\gamma, \alpha_{i'}}'$. Thus, we can ignore the subscripts of $\alpha$ and just denote the hyperplanes as $L_\gamma, L_\gamma'$. 

Choose some $j$. Then for each $\gamma$, there is some $i$ such that $\gamma > \wedge^i \omega$. If $i < j$, then $\alpha_i$ and $\beta_j$ are both in $\Delta_{< j}$, while if $i \geq j$, then they are both in $\Delta_{\geq j}$. In either case, by property 6, the coefficient of $b_j$ in $L_\gamma$ and $L_{\gamma}'$ must either be $0$ or equal to the coefficient of $a_i$. As the coefficients of all of the $a_i$ are either 0 or 1, this implies that the coefficients of all of the $b_j$ are either 0 or 1. By properties 3, 4, and 5, this implies that any hyperplane factor of $b_{\beta_j}$ is either $L_\gamma$ or $L'_\gamma$ for some $\gamma$, or that its coefficients of $a_i$ are all 0.

Let $S_j$, resp. $S_j'$ be the set of $\gamma$ such that the coefficient of $b_j$ in $L_\gamma$, resp. $L_\gamma'$ is 1. Let $S_{i, j}$, resp. $S_{i, j}'$ be the subset of $S_j$, resp. $S_j'$ such that $\gamma > \wedge^i \omega$. As for each $\gamma$ there is some $i$ such that $\gamma > \wedge^i \omega$, $S_j = \cup S_{i, j}, S'_j = \cup S'_{i, j}$. 

Assume without loss of generality that $i < j$, and therefore that $\alpha_i, \beta_j$ are in $\Delta_{i < j}$. Then by property 6, $|S_{i, j}| + |S'_{i, j}| = i(n - j) + i(n - j - 1)$. Also by property 6, if $\gamma \in S_{i, j}$, then $h_\gamma(\rho) + 1 \geq i - j + 1$, while if $\gamma \in S'_{i, j}$, then $h_\gamma(\rho) \geq i - j + 1$. Finally, also by property 6, if $\gamma \in S_{i, j}$ or $S'_{i, j}$, then $\gamma > \wedge^{j - 1} \omega$. Then $\gamma \in S_{i, j}$ or $S'_{i, j}$ implies that $\gamma > \sum_{k = i}^{j - 1} \wedge^k \omega$. 

Assume $\sum_{k = i}^{j - 1} \wedge^k \omega \in S'_{j}$ for some $i$. Then it is also in $S'_{i, j}$; therefore, $\gamma(\rho) \geq i - j + 1$. But $\gamma(\rho) = i - j$, so this is impossible - so no root of the form $\sum_{k = i}^{j - 1} \wedge^k \omega \in S'_{j}$. Therefore, if $\gamma \in S'_{i, j}$, then $\gamma \geq \wedge^j \omega$. 

We therefore have that if $\gamma \in S_{i, j}$, then $\gamma > \wedge^{j - 1}\omega$, while if $\gamma \in S'_{i, j}$, then $\gamma > \wedge^j \omega$. But then $|S_{i, j}| \leq i (n - j)$, while $|S'_{i, j}| \leq i (n - j - 1)$. We know that $|S_{i, j}| + |S'_{i, j}| = i (n - j) + i (n - j - 1)$; therefore, we must have that all of the elements that could be in either must be, and therefore that $S_{i, j} = \{\gamma | \gamma > \wedge^i \omega, \wedge^{j - 1} \omega\}$, while $S'_{i, j} = \{\gamma | \gamma > \wedge^i \omega, \wedge^j \omega\}$. 

Correspondingly, if $i \geq j$ instead, then $S_{i, j} = \{\gamma | \gamma > \wedge^i \omega, \wedge^j \omega\}$ while $S'_{i, j} = \{\gamma | \gamma > \wedge^i \omega, \wedge^{j - 1} \omega\}$. Then as $S_j = \cup S_{i, j}, S'_{j} = \cup S'_{i, j}$, we get that 
$$S_j = \{\gamma | \gamma > \wedge^{j - 1} \omega \enspace \textrm{or} \enspace \gamma > \wedge^j \omega\}$$ and $$S'_j = \{\gamma | \gamma > \wedge^{j - 1} \omega \enspace \textrm{and} \enspace \gamma > \wedge^j \omega\}$$ 

Therefore 
\begin{align*}
L_\gamma &= \sum_{\gamma > \wedge^i \omega} a_i + \sum_{\substack{\gamma > \wedge^{j - 1} \omega \enspace \\ \textrm{or} \enspace  \gamma > \wedge^j \omega}} b_j + h_\gamma(\rho) + 1\\
L_\gamma' &= \sum_{\gamma > \wedge^i \omega} a_i + \sum_{\substack{\gamma > \wedge^{j - 1} \omega \enspace \\ \textrm{and} \enspace \gamma > \wedge^j \omega}} b_j + h_\gamma(\rho)
\end{align*}

This gives $b_{\alpha_i}$ as the product of the $L_\gamma, L_\gamma'$. We now only need to find $b_{\beta_j}$.

By property 6, the degree of $b_{\beta_j}$ is \[{(j - 1)(n - j) + j(n - j - 1) + 1 = (j - 1)(n - j - 1) + (j(n - j) - 1) + 1}\]. But $|S_j'| = (j - 1)(n - j - 1)$, and $|S_j| = j(n - j) - 1$, so only one factor is unaccounted for - and that factor, by property 2, is $b_j + 1$. By inspection, then, we get the formula in the theorem. 
\end{proof}

\section{Appendix}

\begin{theorem}
Let $V$ be a finite dimensional $\mathfrak{b}$-module. Then $(V \otimes \frac{U \mathfrak{g}}{(U \mathfrak{g}) \mathfrak{n}})^{\mathfrak{b}_{-\mu}}$ is a $U\mathfrak{h} = S\mathfrak{h}$ module of rank equal to the dimension of the $\mu$ weight subspace of $V$. 
\end{theorem}
\begin{proof}
As $V$ is finite-dimensional, we can define an increasing filtration by $\mathfrak{b}$-submodules $F_i$ on $V^*$ such that $F_{-1} = \{0\}$, $F_{i + 1}/F_i$ has a unique weight $\lambda_i$ (note that $F_{i + 1}/F_i$ is not necessarily 1-dimensional) , and the $\lambda_i$ are distinct. 

We have an isomorphism$(V \otimes \frac{U \mathfrak{g}}{(U \mathfrak{g}) \mathfrak{n}})^{\mathfrak{b}_{-\mu}} \simeq Hom_{\mathfrak{b}, -\mu}(V^*, \frac{U \mathfrak{g}}{(U \mathfrak{g}) \mathfrak{n}})$ where the subscript denotes that the homomorphism twists by $-\mu$; in other words, if $w \in W$ of weight $\lambda$ and $f \in Hom_{\mathfrak{b}, -\mu}(W, U)$, then $f(w)$ has weight $\lambda - \mu$. This homomorphism set has a residual right $U\mathfrak{h}$ action from its action on $\frac{U \mathfrak{g}}{(U \mathfrak{g}) \mathfrak{n}}$. We proceed by induction on the filtration. 

We can start with the base case of $F_{-1}$, the trivial $\mathfrak{b}$-module; by definition, $Hom_{\mathfrak{b}, -\mu}(F_{-1}, \frac{U \mathfrak{g}}{(U \mathfrak{g}) \mathfrak{n}}) = \{0\}$.

Assume $\lambda_{i + 1} \neq \mu$. Because $F_i$ is a submodule of $F_{i + 1}$, we get a map $\iota_i: Hom_{\mathfrak{b}, -\mu}(F_{i + 1}, \frac{U \mathfrak{g}}{(U \mathfrak{g}) \mathfrak{n}}) \rightarrow Hom_{\mathfrak{b}, -\mu}(F_i, \frac{U \mathfrak{g}}{(U \mathfrak{g}) \mathfrak{n}})$. Kashiwara \cite{KASHIWARA} proved (equation (1.4)) that $\exists g_i \neq 0 \in S \mathfrak{h}$ such that $g_i Hom_{\mathfrak{b}, -\mu}(F_i, \frac{U \mathfrak{g}}{(U \mathfrak{g}) \mathfrak{n}}) \in Im(\iota_i)$,. Therefore $\operatorname{rk} Hom_{\mathfrak{b}, -\mu}(F_{i + 1}, \frac{U \mathfrak{g}}{(U \mathfrak{g}) \mathfrak{n}}) \geq \operatorname{rk} Hom_{\mathfrak{b}, -\mu}(F_i, \frac{U \mathfrak{g}}{(U \mathfrak{g}) \mathfrak{n}})$ as $U \mathfrak{h}$-modules.

Let $f \in Hom_{\mathfrak{b}, -\mu}(F_{i + 1}, \frac{U \mathfrak{g}}{(U \mathfrak{g}) \mathfrak{n}})$ with $\iota_i(f) = 0$. Then $f(F_i) = 0$, so $f$ descends to a map $\bar{f}: F_{i + 1}/F_i \rightarrow \frac{U \mathfrak{g}}{(U \mathfrak{g}) \mathfrak{n}}$ which twists by $-\mu$. For any $v \in F_{i + 1}/F_i$, $\mathfrak{n}v = 0$, so $\mathfrak{n} \bar{f}(v) = 0$. But the only elements of $\frac{U \mathfrak{g}}{(U \mathfrak{g}) \mathfrak{n}}$ with this property are elements of $U \mathfrak{h}$, that is, elements of weight 0. As $\bar{f}$ twists by $-\mu$, and $\lambda_{i + 1} \neq \mu$, this is impossible unless $\bar{f} = 0$, and therefore $f = 0$ - so $\iota$ is injective. Therefore, we get that the ranks must in fact be equal. 

Assume $\lambda_{i + 1} = \mu$. As the $\lambda_i$ are distinct, by the above, $Hom_{\mathfrak{b}, -\mu}(F_i, \frac{U \mathfrak{g}}{(U \mathfrak{g}) \mathfrak{n}})$ is trivial. The map $\iota_i$ is not injective in this case; we need to find its kernel. Choose a basis of $v_j \in F_{i + 1}/F_i$; then for each $j$, we can define $\bar{f}_j: V_{i + 1}/V_i \rightarrow \frac{U \mathfrak{g}}{(U \mathfrak{g}) \mathfrak{n}}$ with $\bar{f}_j(v_k) = 1$ if $j = k, 0$ else. By inspection, these twist by $-\mu$. We can therefore extend these to $f \in Hom_{\mathfrak{b}, -\mu}(F_{i + 1}, \frac{U \mathfrak{g}}{(U \mathfrak{g}) \mathfrak{n}})$, so $\operatorname{rk} Hom_{\mathfrak{b}, -\mu}(F_i, \frac{U \mathfrak{g}}{(U \mathfrak{g}) \mathfrak{n}}) = dim F_{i + 1}/F_i = dim V^{\mu}$. 
\end{proof}

\begin{corollary}
Assume $V = V_\nu^*$ is the dual of a finite-dimensional irreducible representation of $U \mathfrak{g}$ with highest weight $\nu$. If $\nu < \mu$ or $\nu - \mu$ is not integral, then $(V \otimes \frac{U \mathfrak{g}}{(U \mathfrak{g}) \mathfrak{n}})^{\mathfrak{b}_{-\mu}}$ is trivial.
\end{corollary}

\begin{theorem}
Assume that the dimension of $V^\mu$ is 1. Then $(V \otimes \frac{U \mathfrak{g}}{(U \mathfrak{g}) \mathfrak{n}})^{\mathfrak{b}_{-\mu}}$ is free as a right $U \mathfrak{h}$-module.
\end{theorem}

In order to prove this, we will need a lemma about when a submodule of a free submodule is free. Let $A$ be a UFD, and let $N$ be an $A$-module. Define $A^{-1}N = \operatorname{Frac}(A) \tens{A} N$. 
\begin{lemma}
Let $A$ be a UFD, $M$ a free $A$-module, and $N$ a finitely generated submodule of $M$ of rank 1. If $A^{-1}N \cap M = N$, then $N$ is free.
\end{lemma}
\begin{proof}
Note that $A^{-1}N \cap M = \{m \in M|\exists a \enspace s.t. \enspace am \in N\}$. Let $\{n_i\}$ be a minimal generating set of $N$. Assume it contains at least two elements $n, n'$. If we can find some $n'' \in n$ that generates both $n, n'$, then by induction, we can show that $N$ is free. 

Because $N$ is rank 1, there must be some $b, b'$ relatively prime with $b n = b' n'$. Let $\{v_j\}$ be a basis for $M$; then $n = \sum a_j v_j, n' = \sum a'_j v_j$. As they are a basis, we get that $b a_j = b' a'_j$. Then as $A$ is a UFD, $b'|a_j$. Let $a''_j = \frac{a_j}{b'}$, and let ${n'' = \sum a''_j v_j = \frac{n}{b'} = \frac{n'}{b}}$. As $a''_j \in A$ for each $j$, $n'' \in M$, so $n'' \in A^{-1}N \cap M = N$. Therefore, the induction step is done, and $N$ is free. 
\end{proof}

\begin{proof}[Proof of Theorem 18]
Let $N = (V \otimes \frac{U \mathfrak{g}}{(U \mathfrak{g}) \mathfrak{n}})^{\mathfrak{b}_{-\mu}}$; by the first theorem, is a rank 1 $U \mathfrak{h}$-module. As $\frac{U \mathfrak{g}}{(U \mathfrak{g}) \mathfrak{n}}$ is a free right $U \mathfrak{h}$-module, so is $M = V \otimes \frac{U \mathfrak{g}}{(U \mathfrak{g}) \mathfrak{n}}$. We then only need to prove that $N (U \mathfrak{h})^{-1} \cap M = N$. 

By definition, $N (U \mathfrak{h})^{-1} \cap M = \{m \in M|\exists h \in U \mathfrak{h} \; mh \in N\}$. Let $m \in M, h \in U \mathfrak{h}$ such that $mh \in N$. We need to prove two things: that for any $n \in \mathfrak{n}, nm = 0$, and that for any $h' \in \mathfrak{h}, h'm = m (h' - \mu)$.  If $mh \in N$, then $nmh = 0$. But $M$ is free as a right $U \mathfrak{h}$-module, so $nm = 0$. If $mh \in N$, then $h'mh = mh (h' - \mu) = m(h' - \mu) h$. Again as $M$ is free, $h'm = m (h' - \mu)$. Therefore, by the lemma, $N$ is a free rank 1 $U \mathfrak{h}$-module. 
\end{proof}

\end{document}